\documentclass[10pt,a4paper]{article}
\usepackage[utf8]{inputenc}
\usepackage{amsmath,amsthm}
\usepackage{amsfonts,graphicx}
\usepackage{amssymb}
\usepackage{hyperref}

\newtheorem{lemma}{Lemma}
\newtheorem{theorem}{Theorem}
\newtheorem{proposition}{Proposition}
\newtheorem{definition}{Def}

\newtheorem{remark}{Remark}
\newtheorem{example}{Example}
\newtheorem{corollary}{Corollary}
\newcommand{\B}{{\mathcal{B}}}
\newcommand{\NB}{{\mathcal{NB}}}
\renewcommand{\O}{{\mathcal{O}}}
\newcommand{\PO}{{\mathcal{PO}}}
\newcommand{\I}{{\mathcal{I}}}
\newcommand{\Ra}{\mathcal{R}}
\newcommand{\N}{\mathcal{N}} 
\newcommand{\NN}{\mathbb{N}}
\newcommand{\ol}[1]{\overline{#1}}
\newcommand{\LRA}{\Longleftrightarrow}
\newcommand{\xd}{x^\dagger}
\newcommand{\Ha}{H}   

\newcommand{\Fprim}[1]{F'[{#1}]}
\newcommand{\Gprim}[1]{G'[{#1}]}
\newcommand{\Rop}[1]{R_{#1}}
\newcommand{\myx}{x^\dagger}
\newcommand{\sig}{\gamma}  

\newcommand{\Lin}{{{\rm Lin}}}
\newcommand{\ULin}{{{\rm UniLin}}}
\title{Operator ordering by ill-posedness in Hilbert and Banach spaces}
\author{Stefan Kindermann\footnotemark[1]%
\and
Bernd Hofmann\footnotemark[2]%
}

\begin{document}

\maketitle\footnotetext[1]{Industrial Mathematics Institute, Johannes Kepler University Linz, Alternbergergstraße 69, 4040 Linz, Austria. Email: kindermann@indmath.uni-linz.ac.at}
\footnotetext[2]{Chemnitz University of Technology, Faculty of Mathematics, 09107 Chemnitz, Germany.\\ Email: hofmannb@mathematik.tu-chemnitz.de}

\begin{abstract}
For operators representing ill-posed problems, an ordering 
by ill-posed\-ness is proposed, where one operator 
is considered more ill-posed than another one if the former can 
be expressed as a cocatenation of bounded operators involving the latter. 
This definition is motivated by a recent one introduced by Math\'e and Hofmann [Adv. Oper. Theory, 2025] that utilizes bounded and orthogonal operators, 
and we show the equivalence of our new definition with this one for the case of compact and 
non-compact linear operators in Hilbert spaces.  
We compare our ordering with other measures of ill-posedness such as the decay of the singular values, 
norm estimates, and 
range inclusions. Furthermore, as the new definition 
does not depend on the notion of orthogonal operators, it
can be extended to the case of linear operators 
in Banach spaces, and it also provides ideas for applications to nonlinear problems in Hilbert spaces. In the latter context, 
certain nonlinearity conditions can be interpreted as ordering relations between a nonlinear operator 
and its linearization. 
\end{abstract}

\bigskip

{\parindent0em {\bf MSC2020:}}
47A52, 65J20, 47J06, 47B01, 47B02

\bigskip

{\parindent0em {\bf Keywords:}}
ill-posed problem, measures of ill-posedness, singular values, degree of ill-posedness, range inclusion, nonlinearity condition
\bigskip

\section{Introduction} \label{sec:intro}
When studying ill-posed and inverse problems, 
quite frequently one encounters the situation that one problem 
might be ``more difficult" to solve than another, 
and with respect to regularization theory, this  problem is often called ``more ill-posed" than the other. 
Usually, in the case of linear compact operators in Hilbert spaces, the decay rate of the singular values is used as indicator for such discrimination. 

In this article we define an abstract general 
concept of when one problem is more ill-posed than 
another by using an ordering relation based 
on connecting operators. This is motivated 
by a recent definition introduced 
by Math\'e and Hofmann
in \cite{MaHo}, where bounded
and orthogonal operators are used. 

The main point of this article is to define a 
more general ordering than that in 
\cite{MaHo}, which  only employs bounded operators
and thus allows extensions  
to the Banach space case and even the nonlinear case. 
One of the main theoretical results is that
the ordering defined here is 
equivalent to that in \cite{MaHo}, namely both in 
the compact and non-compact case, and in the compact case it is also 
equivalent to the more classical approach of comparing 
the decay rates of the singular values. 

After defining the ordering in Section~\ref{sec:illposedness},  we prove the 
mentioned equivalence in Section~\ref{sec:compact}
in the compact case and in Section~\ref{sec:noncompact}
for the non-compact case. Section~\ref{sec:alternative} is devoted to alternative orderings and the relation to regularization.
Following this, Section~\ref{sec:Douglas} recalls Douglas range inclusion theorem in light of our suggested ordering approach. 
A generalization to the Banach space case 
is done in Section~\ref{sec:Banach}. Furthermore, in Section~\ref{sec:nonlinear}
we extend the ordering ideas to operator equations, 
interpreting certain nonlinearity conditions as  ordering  relation between 
an operator and its linearization.

\section{Orderings by ill-posedness} \label{sec:illposedness}
Let $A$ and $A'$ be bounded linear operators $A: X \to Y$ and $A': X^\prime \to Y^\prime$.
We use the following definitions: 
$\N(A)$ denotes the nullspace of $A$ and 
$\Ra(A)$ its range. 
$\B(X,Y)$ denotes the set of {\em bounded} operators between Banach spaces $X$ and $Y$. If the spaces 
are clear from the context, then we simply write $\B$. 
 $\O(X,Y)$ denotes the set of {\em isometric} operators between {\em Hilbert} spaces $X$ and $Y$. If the spaces 
are clear from the context, then we simply write $\O$. 
Moreover we denote by $\PO(X,Y)$ the set of {\em partial isometries}.

Note that an isometric operator $U$
satisfies $U^TU = I$. If it is additionally surjective $\Ra(U) =Y$, then we speak of an {\em orthogonal} (in real spaces) 
or {\em unitary} (in complex spaces)
operator.  An operator $U$ that has a nullspace but 
is an isometry when restricted to  $\N(U)^\bot$ is a {\em partial isometry}.

An ordering of inverse problems by ill-posedness 
was proposed in \cite{HoKi10},
where a measure of ill-posedness was abstractly defined as an ordering of operators. 
In this article, motivated by  \cite{MaHo},
the following ordering is the main tool that we work with: 
\begin{definition}\label{def:main}
Let $X,Y,X',Y'$ be Banach spaces. 
Let $A: X\to Y$, $A': X'\to Y'$ be bounded linear operators. 
The operator $A'$ is said to be \emph{more ill-posed} than the operator $A$ if there applies an ordering 
defined as 
\begin{align}\label{eq:threefactors}
\begin{split} &A' \leq_{\B,\B} A\\
\Longleftrightarrow \qquad &\exists \,T \in \B(\ol{\Ra(A)},Y')\;\;\mbox{and}\;\;\exists\, 
S \in \B(X,X') \quad   \text{such that} \quad   A' = T A S. 
\end{split}
\end{align}
If 
\[ A' \leq_{\B,\B} A \qquad \text{ and }  \qquad  A \leq_{\B,\B} A',  \]
then both operators are \emph{equivalent with respect to ill-posedness}, and we write 
\[ A '\sim_{\B,\B} A. \] 
If either $A^\prime$ is more ill-posed than $A$ or $A$ is more ill-posed than $A^\prime$, then $A$ and $A^\prime$ are said to be \emph{comparable}, otherwise \emph{non-comparable}. If $A^\prime$ is more ill-posed than $A$, but $A$ fails to be more ill-posed than $A^\prime$, then $A^\prime$ is said to be \emph{strictly more ill-posed} than $A$.

The operators $T,S$ in \eqref{eq:threefactors}
are referred to as \emph{connecting operators} 
in the ordering. 
\end{definition} 
As we will see, in many cases this  definition reflects  most of the established notions of 
$A'$ is ``more ill-posed'' than $A$.

\section{Compact operators in Hilbert spaces} \label{sec:compact}
\subsection{General assertions} \label{subsec:generalcom}
Let us specialize the definition to the simplest situation when all spaces are Hilbertian and 
the operators $A,A'$ are compact. 
In the Sections~\ref{sec:compact} to \ref{sec:Douglas}
we assume throughout that 
\begin{equation}\label{Hilbert} X,Y,X',Y'  \text{ are  Hilbert spaces.} 
\end{equation}

\begin{proposition}
 A simple consequence of Definition~\ref{def:main} is the adjoint-invariance of the ordering: 
 \[ A' \leq_{\B,\B} A \Longleftrightarrow {A'}^* \leq_{\B,\B} A^* .  \]
\end{proposition}

Under the stated assumption, we may use seemingly (stronger) alternative definitions of ordering that uses isometric operators instead of bounded ones, as it 
was used in \cite{MaHo}. 
\begin{definition}\label{def1}
Define the following orderings: 
\begin{alignat*}{3} A' &\leq_{\O,\B} A & &\Longleftrightarrow \ \exists\, O \in \O(\ol{\Ra(A)},Y'),\; S \in S \in \B(X,X')& \; &\text{ with }\quad   A' = O AS\,, \\
A' &\leq_{\B,\O} A & \ &\Longleftrightarrow \ \exists\, O \in \B(X,X'),\; T \in \O(X,\N(A)^\bot)& \; &\text{ with }\quad   A' = T A O \,.
\end{alignat*}
If the operators in $\O$ can be chosen as identity, 
then we write $\leq_{\I,\B}$ or $\leq_{\B,\I}$. 

Moreover, in case that $A,A'$ are both compact, we may define an ordering by the singular values
 \[ A' \leq _{\sigma} A\; \Longleftrightarrow  \;\forall \, n \in \mathbb{N}:\; \sigma_n(A') \leq  \sigma_n(A) \]
 and similarly,
  \[ A' \lessapprox _{\sigma,C} A\; \Longleftrightarrow \;  \exists \, C >0 \;\;\text{such that}\;\;\forall \, n \in \mathbb{N}:\; \sigma_n(A') \leq  C \sigma_n(A). \]
\end{definition}

Clearly, $A' \leq_{\O,\B} A$ or  $A' \leq_{\B,\O} A$ imply $A' \leq_{\B,\B} A$. 
But we will show below that all three orderings are in fact equivalent. 
For this we need the results of the following lemma, and we refer in this context also to \cite[Prop.~3]{MaHo}. 
\begin{lemma}\label{lem:MaHo}
Let $A:X\to Y$, $A':X' \to Y'$ be compact operators acting between Hilbert spaces. 

\begin{itemize} 
 \item
In case that  $\dim(\Ra(A')) =\dim(\Ra(A))$ (finite or not), 
we have 
\[ A'\lessapprox_{\sigma,C} A \Rightarrow A' \leq_{\O,\B} A ,\] 
with $O:\ol{\Ra(A)} \to \ol{\Ra(A')}$ unitary. Also, we have 
\[ A'\lessapprox_{\sigma,C} A \Rightarrow A' \leq_{\B,\O} A, \]
with $O:\N(A')^\bot \to \N(A)^\bot$ unitary. 
\item In case that  $\Ra(A')$ is finite-dimensional and  $\dim(\Ra(A')) <\dim(\Ra(A))$,
we have 
\[ A'\lessapprox_{\sigma,C} A \Rightarrow A' = O A S, \] 
where $S$ is bounded and $\O: \ol{\Ra(A)} \to  \ol{\Ra(A')}$ is a partial 
isometry. 
\end{itemize}
\end{lemma}
\begin{proof}
Recall the singular value decomposition of compact operators:
\[ A' = \sum_{i \in I'} \sigma_i' (.,\phi_i') \psi_i' \qquad 
 A = \sum_{i \in I} \sigma_i (.,\phi_i) \psi_i. \]
Here, $\sigma_i,\sigma_i'>0$ are the positive singular values of $A$ and $A'$, respectively.  
The functions $\psi_i,\psi_i'$ are  orthogonal bases of $\ol{R(A)}, \ol{R(A')}$, respectively, 
and $\phi_i,\phi_i'$ are  orthogonal bases of $\N(A)^\bot, \N(A')^\bot$, respectively. 
The index sets $I,I'$ are either countably infinite or finite depending on 
the dimension of the ranges, or equivalently, on the number of nonzero singular values. 
In any case by $A'\lessapprox_{\sigma,C} A$ we can assume
that  $I' \subset I$. 

Define the following operators:
\begin{equation} \label{eq:Smap} S f:= \sum_{i \in I'} \frac{\sigma_i'}{\sigma_i} (f,\phi_i') \phi_i  \end{equation}
\[ O f:= \sum_{i \in I'} (f,\psi_i) \psi_i'.  \]
Then, inserting the definitions, yields (the indices $i,j,k$ correspond to $O$,$A$,$S$, respectively)   
\begin{align*}
O A S f &= 
\sum_{i \in I'} 
\sum_{j \in I} 
\sum_{k\in I'} \sigma_{j} \frac{\sigma_k'}{\sigma_k} (f,\phi_k') (\phi_k, \phi_j)  (\psi_j,\psi_i)  \psi_i'\, .
\end{align*}
By orthogonality 
$(\phi_k, \phi_j) = \delta_{k,j}$ and $(\psi_j,\psi_i) = \delta_{i,j}$, and thus two summations drop out: 
\begin{align*}
O A S f &= 
\sum_{i \in I'} 
 \sigma_{i} \frac{\sigma_i'}{\sigma_i} (f,\phi_i')   \psi_i' = 
\sum_{i \in I'}  {\sigma_i'} (f,\phi_i')   \psi_i' = A' f.
\end{align*}
By $A'\lessapprox_{\sigma,C} A$ it is easy to show that $S$ is bounded. 
In case that  $\dim(\Ra(A')) =\dim(\Ra(A))$, we have that $I' = I$, 
and then it follows easily that 
$O$ is an isometry $\ol{R(A)} \to \ol{R(A')} \subset Y'$ and surjective, hence unitary. 
Since $A'\lessapprox_{\sigma,C} A$ clearly implies ${A'}^*\lessapprox_{\sigma,C} A^*$, 
applying the previous results gives ${A'}^* = O A^* S$, hence 
$A' = S^* A O^*$, where $O^*: \ol{R(A')}= \N(A')^\bot \to  \ol{R(A)}  = \N(A)^\bot$ is 
again unitary.

In the other case $\dim(\Ra(A')) <\dim(\Ra(A))$,  we have $I' < I$, 
and $O$ is only a partial isometry since the elements $\psi_i$, $i\in I\setminus I'$ are 
in its nullspace. 
\end{proof}

\begin{remark} \rm
The case that  $\dim(\Ra(A')) <\dim(\Ra(A))$ cannot appear if $A$ and $A'$ are both ill-posed operators in the sense of 
Nashed \cite{Nashed}. Since we are mostly interested in this situation, we do not consider it in detail.
\end{remark}

A simple corollary of Lemma~\ref{lem:MaHo} is the following, 
which is obtained  since $S$ in~\eqref{eq:Smap} is easily seen to be 
invertible under the following conditions: 
\begin{corollary}\label{cor:iso}
Assume that $A,A'$ are compact and injective operators between 
Hilbert spaces with equivalent decay rates of the singular values, i.e., 
$A \approx_{\sigma} A'$, or in more detail,  
\begin{equation}
\exists \,c,C: \quad 
c \,\sigma_n(A') \leq \sigma_n(A) \leq C\, \sigma_n(A') \quad \forall n \in \mathbb{N}. 
\end{equation}
Then there exists an isomorphism $S$ (i.e., bounded invertible linear map)  and a unitary operator $U$
such that 
\[ A'  = S A U \quad \text{ or} \quad  A' = U A S. \]
\end{corollary}

Now we are ready to state the first main result of the equivalence of 
 Definition~\ref{def:main} and Definition~\ref{def1}:
\begin{theorem}\label{th0}
Let $A,A^\prime$ be compact operators between Hilbert spaces such that 
both of them are ill-posed, which means that $\dim(\Ra(A)) = \dim(\Ra(A')) = \infty$. 

Then the following statements are equivalent:  
\begin{enumerate} 
\item\label{one} \[ A' \leq_{\B,\B} A\,.\]
\item\label{two} \[ A' \leq_{\O,\B} A\,. \]
\item\label{three} \[ A' \leq_{\B,\O} A\,. \]
\item\label{four} \[ A' \lessapprox_{\sigma,C}  A\,. \]
\end{enumerate} 
\end{theorem}
\begin{proof}
Assume that item~\ref{one} holds. Then, $\sigma_n(A') \leq \|T\|\|S\| \sigma_n(A)$, $\forall n \in \NN$, 
and hence, item~\ref{four} is also valid with 
$C =  \|T\|\|S\|$.  By Lemma~\ref{lem:MaHo}, we have that 
items~\ref{two} and \ref{three} hold, and this clearly implies item~\ref{one}. Thus, all statements are equivalent. 
\end{proof}

As a corollary we have a stability result of  our ordering: 
\begin{corollary}\label{cor:stab}
Let $A_n, A'_n$ be sequences of compact operators that converge in norm
to ill-posed limit operators as $A_n \to A$ and $A_n'\to A$ as $n \to \infty$, and where the ordering  
\[ A_n' \leq_{\B,\B} A_n \qquad \forall n \in \NN \] 
is satisfied. 
Then the limits preserve the ordering 
\[ A' \leq_{\B,\B}  A.\]
\end{corollary}
\begin{proof}
It follows directly that $\sigma_k(A_n) \leq \sigma_k(A_n')$ for all $k$ and $n$, and since 
the singular values are continuous, we have $ A' \lessapprox_{\sigma,C}  A$. Consequently, the result follows by applying Theorem~\ref{th0}. 
\end{proof}

\begin{remark} \rm
Let us point out an apparent paradox with our definition of ordering as ``more ill-posed''. 
Taking in $A' = T A S$ for $T$  an operator with finite-dimensional range, it follows that 
$A' \leq_{\B,\B} A$ holds, but since $A'$ has finite-dimensional range, it corresponds to a well-posed operator. 
In this sense any such well-posed operator with finite-dimensional range is ``more ill-posed'' than an ill-posed operator. 
The same problem does  appear for other orderings like $\lessapprox_{\sigma,C}$ as well. The resolution  of the ``paradox'' is that 
any such finite-dimensional operator contains in a neighborhood  an ill-posed operators of arbitrary high degree of ill-posedness. 
Thus, as soon as we would like to have  a certain stability property of the ordering, as in Corollary~\ref{cor:stab}, such a paradox is unavoidable. 
\end{remark}

\subsection{Representation of moderately ill-posed operators} \label{subsec:moderate}
An interesting consequence of the above results is a characterization 
of moderately ill-posed operators with \emph{degree of ill-posedness} $k \in \mathbb{N}$. By this we mean a compact 
operator showing a power-type decay of singular values with the exponent $-k$. 
Basically, the canonical example of such an operator is 
$k$-times integration $I_k: L^2(0,1) \to L^2(0,1)$ with  $I_k = J^k$, where the simple integration operator $J: L^2(0,1) \to L^2(0,1)$  
is defined as $J:  f \to \int_0^x f(t) dt$. This definition 
can be extended to noninteger $k$ by using fractional integration.
The moderate ill-posedness of $I_k$ with degree $k$ is a direct consequence of the well-known decay rate
$\sigma_n(I_k) \sim n^{-k}$ of the singular values of $I_k$. Now by Corollary~\ref{cor:iso}
every injective moderately ill-posed operator with such singular value decay 
is isomorphic to the $k$-times integration operator:
\begin{theorem}
Let $A$ be a compact and injective linear operator mapping between infinite-dimensional Hilbert spaces, and 
assume that we have, for some $k \in \mathbb{N}$,
\begin{equation}\label{eq:id} \sigma_n(A) \sim \frac{1}{n^k}\,.\end{equation}
Then there are isomorphic linear operators $S$ and $T$ such that 
\[ A = S J^k T.\]
One of the operators $S$ or $T$ can be chosen unitary. 
\end{theorem}
Thus, the theory of moderately ill-posed problem boils down 
to understanding differentiation. 

\smallskip

This theorem can be extended to more general situations, keeping 
the main assumptions on $A$. With minor modifications we may consider instead of \eqref{eq:id} more general 
decay rates: For $\alpha > \beta $ and constants $c,C>0$:
\begin{equation}\label{eq:id2}  \frac{c}{n^\alpha} \leq  \sigma_n(A) \leq  \frac{C}{n^\beta} \, . 
\end{equation}
In this case we have 
\[ A = S J^\beta T,\]
where one of the operators $S,T$ being unitary and the other 
one, say $S$, satisfies the error estimate
\[ c \|J^{\alpha-\beta} x\| \leq \|S x\| \leq C \|x\|. \]
Note that instead of $J$, other operators might be used, for example Sobolev scale embeddings.

The estimate \eqref{eq:id2} implies that 
the {\em interval of ill-posedness} (see \cite{HoTa97}) 
defined as 
\begin{equation}\label{intill} [\underline{\mu},\overline{\mu}] := 
\left[ \liminf_{n\to  \infty} \frac{-\log(\sigma_n(A))}{\log(n)}, 
\limsup_{n\to  \infty} \frac{-\log(\sigma_n(A))}{\log(n)}
\right]
\end{equation}
satisfies  $ [\underline{\mu},\overline{\mu}] = [\beta,\alpha]$.
Conversely it 
has shown in \cite{HoTa97} that a finite interval of ill-posedness $[\underline{\mu},\overline{\mu}]$ implies   \eqref{eq:id2}  
with $\beta = \underline{\mu}-\epsilon$, $\alpha = \overline{\mu} +\epsilon$
for all $\epsilon>0$. Note that in this case the lower bound $\underline{\mu}$ characterizes 
the degree of ill-posedness.

\section{Alternative orderings and the relation to regularization} \label{sec:alternative}
Let us recall some alternative orderings that were defined in \cite{HoKi10}:
\begin{alignat*}{3}
 A' &\leq_{\text{norm}} A& \  &\Longleftrightarrow \quad \|A' x\| \leq \| A x\|& \quad &\forall x \in X.  \\ 
  A' &\leq_{\text{norm},C} A& &\Longleftrightarrow \quad \|A' x\| \leq C \| A x\|& \quad &\forall x \in X.   
\end{alignat*}
Let $M \subset X$ be a conical set. 
Define the {\em modulus of injectivity} as 
\[ j(A,M):= \inf_{x \in M, x\not= 0} \frac{\|A x\|}{\|x\|}. \] 
The modulus of injectivity is related to the 
{\em modulus of continuity} 
\[ \omega(\delta, A,M):= \sup\left\{\|x\| \, :\, x \in M, \|A x\| \leq \delta \right\} \]
by  the identity \cite{HoKi10}:
\[ j(A,M) = \frac{\delta}{\omega(\delta, A,M)}. \]
Thus, we may define the ordering by continuity (or injectivity)  as follows:
Let $M_\gamma$ be a family of increasing conical sets with $\bigcup_\gamma M_\gamma = X$.  Then,
\[ A' \leq_{j,M_\gamma} A \quad \Longleftrightarrow  \quad j(A,M_\gamma) \leq j(B,M_\gamma)  \quad \forall \gamma .  \]  
It follows immediately that 
\begin{align}  A' \leq_{\text{norm}} A \Rightarrow A' \leq_{j,M} A . \end{align}
Of particular interest in applications and for discretization, we let 
$M_\gamma = X_n$, where $X_n$  is a strictly increasing sequence of finite-dimensional subspaces with \mbox{$\dim(X_n) = n$}. 
We have \cite{HoKi10}
\[  A' \leq_{\text{norm}} A \Rightarrow A' \leq_{j,X_n} A \Rightarrow A' \leq_{\sigma} A  . \]
The opposite direction is obtained by replacing $A$ by $A O$: 
\[ A'\leq_{\sigma} A  \Rightarrow  \exists O \in \O: \forall (X_n)_n: j(A',X_n) \leq j(A,O X_n), \]
or, 
\[ A'\leq_{\sigma} A  \LRA \forall (X_n)_n: \exists Y_n \,  j(A',X_n) \leq j(A,Y_n). \]

One motivation for studying orderings of ill-posed operators comes from comparing approximation 
rates for (Tikhonov) regularization schemes. Assume that two linear ill-posed problems with the same 
exact solution $\xd$ are modeled by the operators $A$ and $A'$, respectively. 
That means, we consider regularized solutions to $A \xd= y$ and $A' \xd= y'$. 
For simplicity we consider exact data $y = A \xd$ and $y' = A \xd$. 
Using Tikhonov regularization $x_{A,\alpha}:= (A^*A  + \alpha I)^{-1}A^* y$ and 
$x_{A',\alpha}:= ({A'}^*A'  + \alpha I)^{-1}{A'}^* y'$, we may compare the approximation 
errors between the regularized solutions and $\xd$.  It makes sense to define  
\[ A' \leq_{Tik} A \Longleftrightarrow  \|x_{A',\alpha} -\xd\|  \geq \|x_{A,\alpha} -\xd\|  \qquad \forall \xd, \forall \alpha>0, \]
which means that $A'$ is more ill-posed in this ordering if the approximation error is always larger than the 
``less'' ill-posed problem with $A$. 

The following relation for the above ordering was shown in \cite{HoKi10}:
\[  A' \leq_{Tik} A  \Longleftrightarrow  {A'}^*A' \leq_{norm} A^*A.  \]

\section{Douglas range inclusion theorem} \label{sec:Douglas} 
A quite useful tool for studying operator factorization is the 
Douglas range inclusion theorem \cite{Doug}, which was employed recently in \cite{MaHo} and which we recall here in form of Theorem~\ref{DRT} for our concept. So we can introduce 
an ordering  related to Douglas' theorem defined by a range inclusion as follows: 
For operators $A',A$ with common images space $Y = Y'$, 
we define 
\[ A' \leq_{R} A \Longleftrightarrow \Ra(A') \subset \Ra(A).  \]
See also \cite{BHTY} for range inclusions with index functions.

\begin{theorem}[Douglas range inclusion 
Theorem]\label{DRT}
The following statements are equivalent:  
\begin{enumerate} 
 \item \[  \Ra(A') \subset \Ra(A) .\] 
\item \[  \exists C>0: \|{A'}^* y\| \leq C \|A^*y\|          .    \]
\item \[ \exists S \in \B \qquad A' =  A S .\]
 \end{enumerate} 
The operator $S$ can be chosen as 
\begin{equation}\label{seq} S =  A^\dagger A'. \end{equation}
If we impose that $\Ra(S) \subset \N(A)^\bot$, then $S$ is uniquely defined by \eqref{seq}.
Moreover, under this condition $S|_{\N(A')^\bot}$ is injective. 
\end{theorem}
\begin{proof}
The proof of the equivalences appeared in \cite{Doug}. The choice of $S$ was as in \eqref{seq}, and 
it was shown that $S$ has a closed graph and is hence continuous. 
We show the uniqueness of $S$. If an arbitrary such $S$ maps into $\N(A)^\bot$, then 
by using $A^\dagger$ and considering the Moore-Penrose equation of the form $A^\dagger A  = P_{N(A)^\bot}$, with 
$P_{N(A)^\bot}$ the orthogonal projector onto $N(A)^\bot$, we get
$A^\dagger A' = A^\dagger A S = P_{N(A)^\bot} S = S$. 
Thus $S$ must have the structure \eqref{seq}. By applying $A$ to \eqref{seq} 
it follows easily that $\N(S) \subset N(A')$.  
\end{proof}


By Theorem~\ref{DRT} we obtain
\[ A'\leq_{\Ra} A  \Longleftrightarrow  {A'}^* \leq_{\text{norm},C} A^* \Longleftrightarrow A' = A S \qquad \Rightarrow A'\leq_{\O,\B} A, \]
as well as 
\[  {A'}^*\leq_{\Ra} A^*  \Longleftrightarrow   {A'} \leq_{\text{norm},C}  A \Longleftrightarrow A' = T A  \qquad \Rightarrow A'\leq_{\B,\O} A. \]
The opposite direction is immediate: 
\[ A'\leq_{\B,\B} A \Longleftrightarrow \exists Q \in \O: 
A' \leq_{\text{norm},C} A Q   \Longleftrightarrow \Ra(Q^*{A'}^*) \subset \Ra(A^*). \]
  
A schematic description of the different relations are given in Table~\ref{tab1}.
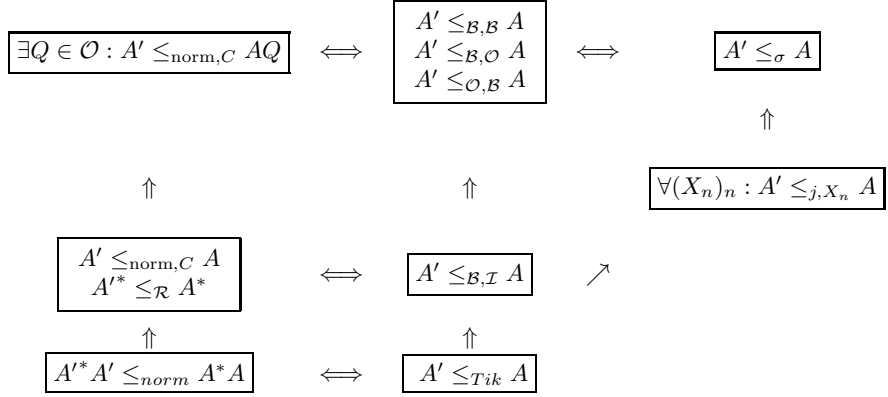
\begin{table}[h!]
\caption{Relation between different orderings.}\label{tab1} 
\resizebox{\textwidth}{!}{ %
\begin{tabular}{ccccc} 
\fbox{$\exists Q\in \O: 
A' \leq_{\text{norm},C}  A Q $}&$\LRA$& 
\fbox{$\begin{array}{c} 
A'\leq_{\B,\B} A   \\
A'\leq_{\B,\O} A   \\
A'\leq_{\O,\B} A   
\end{array}$}
&$\LRA$&  \fbox{$A'\leq_{\sigma} A $} \\[5mm]
 & & & &  $\Uparrow$  \\[5mm]
 $\Uparrow$ &  & $\Uparrow$
 & &  \fbox{$\forall (X_n)_n: A'\leq_{j,X_n} A $} \\[5mm]
  \fbox{$\begin{array}{c} A' \leq_{\text{norm},C}  A \\  {A'}^* \leq_{\Ra} A^*  \end{array} $}&$\LRA$ &
 \fbox{$A' \leq_{\B,\I} A$} & $\nearrow$ &  \\[5mm]
    $\Uparrow$ &  & $\Uparrow$ \\
  \fbox{${A'}^*A' \leq_{norm} A^*A $}&$\LRA$ & \fbox{ $A' \leq_{Tik} A$}\\
\end{tabular}
}
\end{table}

\section{Ordering in the non-compact case}\label{sec:noncompact}
\subsection{General assertions} \label{subsec:generalnoncom} 
The previous results for compact operators made heavy use of the 
singular value decomposition. We next address the question how far 
they can be generalized to the non-compact case. 
It turns out that the main Theorem~\ref{th0} is still valid. 

At first we state the following well-known results \cite[Prop. 2.18]{EHN}:
\begin{proposition}\label{ric}
Let $A$ be bounded between Hilbert spaces. Then,
\[ \Ra(A^*) = \Ra(\sqrt{A^*A}).\]
\end{proposition} 

A consequence is the polar decomposition; see, \cite[p.~323]{rudin} and the remark ibid. 
\begin{proposition}
Let $A$ be bounded. Then there exists a $U$ which is an partial isometry such that 
\[ A = U \sqrt{A^*A}, \]
where 
\[ U:  \ol{R(A^*)} \to \ol{R(A)} \]
is an isometry.
If there is an isometry from $N(A)$ onto $N(A^*)$, 
then $U$ is unitary. 
\end{proposition}

We have the following lemma:
\begin{lemma}\label{lemmahelp}
Let $A,R \in \B$. 
Then there exists a partial isometry $Q :\ol{R(A)} \to \ol{\Ra(RA)}$, and an 
$S \in B$
with 
\[ RA = Q A S.\]
If 
\begin{equation}\label{thisc} \ol{\Ra(A^*)} = \ol{\Ra(A^* R^*)},  \end{equation}
then $Q$ is an isometry.
\end{lemma}
\begin{proof}
By the polar decomposition 
%
%
we have
  an  isometry $U$ from $\ol{\Ra(A^*R^*)}  \to \ol{\Ra(RA)}$ with 
\[ RA = U \sqrt{A^*R^*R A}.\]
Let $Z:= \sqrt{A^*R^*R A}$. 

Thus, 
\[ \Ra(Z) = \Ra(A^*R^*) \subset \Ra(A^*) = 
\Ra(\sqrt{A^*A}). \]
By the Douglas range inclusion theorem, it follows that 
there exists an $S \in \B$ with 
\[ Z = \sqrt{A^*A} S.\]
Again by the polar decomposition we have an isometry 
 $W:\overline{\Ra(A^*)} \to \overline{\Ra(A)}$ with 
$A = W\sqrt{A^*A}$.
Now $W^*W  = I$ when restricted to  $\overline{\Ra(A^*)}$. 
Thus, 
\[ W^*A = \sqrt{A^*A}, \qquad  R A  = U \sqrt{A^*A} S = 
U W^* A S,\]
and $U W^*$ is a partial isometry from $\ol{\Ra(A)} \to \ol{\Ra(RA)}$ as claimed.

Assume that \eqref{thisc} holds, and let $x \in \N(U W^*)$, i.e., $U W^* x = $ 
Then $\Ra(W^*) = \ol{\Ra(A^*)} = \ol{\Ra(A^*R^*)}$. Since $\N(U)  =\Ra(A^*R^*)^\bot$ 
it follows that $W^* x = 0$. Since $\N(W^*) = 
\Ra(A)^\bot$, the result follows.   
\end{proof}
The condition \eqref{thisc} can equally be written as 
\[ \N(A) = \N(RA). \] 

\begin{theorem}\label{th00}
Let $A,A^\prime$ be bounded between Hilbert spaces. 

Then the following statements are equivalent:  
\begin{enumerate} 
\item\label{Aone} \[ A' \leq_{\B,\B} A\,.\]
\item\label{Atwo} \[ A' \leq_{\PO,\B} A\,. \]
\item\label{Athree} \[ A' \leq_{\B,\PO} A\,. \]
\end{enumerate} 

If the connecting operator $T$ in item~\ref{Aone} is injective on $\ol{R(A)}$, then we have  
\[ A' \leq_{\B,\B} A \Longleftrightarrow A' \leq_{\O,\B} A  .\] 
If $\Ra(S)$ in item~\ref{Aone} is dense in $\N(A)^\bot$, then 
\[ A' \leq_{\B,\B} A \Longleftrightarrow A' \leq_{\B,\O} A . \] 
\end{theorem}
\begin{proof}
Assume that item~\ref{Aone} holds.  Then $A' = T A S$, and by Lemma~\ref{lemmahelp} applied to $TA$ 
we have  $T A S = Q A \tilde{S} S$, where $Q$ is a partial isometry. Since $\tilde{S} S$ is bounded, 
item~\ref{Atwo} follows. The same with ${A'}^*, A^*$ in place of $A',A$ yields item~\ref{Athree}.

If $T$ is injective, then $\N(T A) = \N(A)$, and by the statement above, $Q$ is an isometry. 
The same argument with adjoints and $S$ yields the second statement.  
\end{proof}

\subsection{Composition of Hausdorff and  Ces\`{a}ro operators} \label{subsec:Hausdorff}
In \cite{HoKi24}, we have considered compositions of the non-compact Hausdorff operator  $\Ha: L^2(0,1) \to \ell^2(\NN)$ with non-closed range defined as
\begin{equation} \label{eq:H}
[\Ha\,x]_j:=\int _0^1 x(t)\, t^{j-1} dt \qquad (j=1,2,\ldots, \quad x \in L^2(0,1))\,,
\end{equation}
possessing the corresponding adjoint operator $\Ha^*:\ell^2(\NN) \to L^2(0,1)$ of the form
\begin{equation} \label{eq:Hstar}
[\Ha^*y](t):= \sum _{j=1}^\infty y_j\, t^{j-1} \qquad (0 \le t \le 1,\quad y \in \ell^2(\NN)),
\end{equation}
and the non-compact Ces\`{a}ro operator $C:L^2(0,1) \to L^2(0,1)$ with non-closed range defined as
\begin{equation} \label{eq:C}
[C\,x](s):= \frac{1}{s}\,\int_0^s x(t)\,dt \qquad (0 \le s \le 1, \quad x \in L^2(0,1))\,,
\end{equation}
having $C^*:L^2(0,1) \to L^2(0,1)$ of the form
\begin{equation} \label{eq:Cstar}
 [C^*x](t):= \int_t^1 \frac{x(s)}{s}\,\,ds \qquad (0 \le t \le 1, \quad x \in L^2(0,1))
\end{equation}
as adjoint operator. Generalizations of these operators in 
weighted spaces  and
their ill-posedness are discussed in \cite{Kindermann24}.

There is a connection  with the compact self-adjoint diagonal operator $D: \ell^2(\NN) \to \ell^2(\NN)$ defined as
\begin{equation} \label{eq:D}
[D\,y]_j:=\frac{y_j}{j}  \qquad (j=1,2,\ldots,\quad y \in \ell^2(\NN))
\end{equation}
of the form $D H = H C^*$ (see \cite[Prop.~2]{HoKi24}), but our focus here is on the adjoint version with the composition  $H^*D: \ell^2(\NN) \to L^2(0,1)$ leading to the equality
\begin{equation} \label{eq:sim1}
 H^*\,D \,=\, C \, H^* \,.
\end{equation}
Also here, the composition  $H^* D$ is a compact operator even though both factors $C$ and $H^*$ are non-compact operators.

We mention two facts in this context. On the one hand, the equation \eqref{eq:sim1} expresses some kind of \emph{similarity} between the operators $D$ and $C$, even if one cannot simply write $D=(H^*)^{-1}C H^*$, because $(H^*)^{-1}$ is an unbounded operator.
For the factorization \eqref{eq:sim1}, relation \eqref{eq:threefactors} applies with $A:=C$, $A^\prime=H^{*}D$, $X=L^2(0,1),X^\prime=\ell^2(\NN)$ and $Y=Y^\prime=L^2(0,1)$  as 
$$H^{*} \,D \prec_{\I,\B} C , $$ 
with the identity operator $I$ in $L^2(0,1)$ and
the connecting operator in $\B(X,X^\prime)$ is $S:=H^*$. In the sense of Definition~\ref{def:main}, the compact operator $H^{*}D$ is \emph{strictly more ill-posed} than the non-compact Ces\`{a}ro operator $C$, because non-compact operators can never be more ill-posed than compact operators. This is a consequence of the factorization in \eqref{eq:threefactors},
where a compact operator $A$ on the right-hand side always implies compactness of $A^\prime$.

For ill-posedness degree discussions of compositions for the Hausdorff operator, the Ces\`{a}ro  operator, and multiplication operators with the integration operator $J$ we refer to  \cite{Gerth21,HoMa22}, \cite{DFH24}, and \cite{HofWolf05,HofWolf09}, respectively.

\subsection{Non-compact multiplication operators mimicking compact operators} \label{subsec:mimic}
As outlined in \cite{WH24}, the Halmos spectral theorem can be helpful for measuring and comparing the ill-posedness of classes of injective, positive semi-definite, self-adjoint, and bounded linear operators with non-closed range by using orthogonal transformations leading 
to bounded linear multiplication operators $M_\lambda$ mapping in the real Hilbert space $L^2([0,\infty))$ with non-closed range defined as
\begin{equation} \label{eq:Multop}
[M\,x](\omega):=\lambda(\omega)\,x(\omega) \qquad (0 \le \omega < \infty, \quad x \in L^2([0,\infty))),
\end{equation}
for real multiplier functions $\lambda \in L^\infty([0,\infty))$. 
Here, for simplicity 
we consider  only the case that the spectral theorem leads to the measure space $(\Omega,\mu)=([0,\infty),\mu_{Leb})$ with the Lebesgue measure
leading to 
the Hilbert space $L^2([0,\infty)$ as basis space 
for $M$. 
Using the Lebesgue measure  seems to be reasonable, and we refer to discussions and examples of \cite{WH24} in this context. 


Using multiplication operators, the 
point of this section is the observation that, 
in contrast to the compact case, cf.~Theorem~\ref{th0}, the spectrum 
of, say, $A^*A$ alone 
does not contain enough information to conclude 
about equivalences with respect to the ordering 
$\leq_{\B,\B}$.  Specifically, 
we construct operators $A'$ and $A$, where 
$A$ is compact and $A'$ not, which have the same 
spectrum but are not equivalent with respect 
to $\leq_{\B,\B}$.


For that purpose, we consider the diagonalized injective self-adjoint compact operator $A^\prime:\ell^2(\NN) \to \ell^2(\NN)$ with non-increasingly ordered singular values
$\sigma_n^2 \;(n=1,2,\ldots)$ tending to zeros as $n \to \infty$, which will be introduced as
\begin{equation} \label{eq:compact}
[A^\prime z]_n:=\sigma_n^2\,z_n \quad (n=1,2,\ldots, \quad z=(z_1,z_2,\ldots)^T \in \ell^2(\NN)).
\end{equation}
Its mimicking multiplication operator counterpart is the operator
\begin{equation} \label{eq:mimicking}
[A x](\omega):=\sum \limits_{n=1}^\infty \sigma_n^2 \,\chi_{[n-1,n)}(\omega) x(\omega)\quad (0 \le \omega <\infty, \quad x \in L^2([0,\infty)),
\end{equation}
defined as a sum of characteristic functions over sets of Lebesgue measure one. 

The spectrum $\sigma(A)$ of the operator $A$ equals the essential range 
of the multiplier function $\lambda$ (see~\cite[Theorem~2.1(g)]{Haase17}), and this gives 
\[ \sigma(A) = 
\{0\} \cup \bigcup_{n=1}^\infty \{\sigma_n^2\}
= \sigma(A^\prime) .
\]
One could conjecture that then also $A$ is a compact operator, but this is not true, because all eigenvalues $\sigma_n^2$ of $A$
possess infinite-dimensional eigenspaces, and we refer to \cite[Prop.~2]{WH24} for the following proposition:
\begin{proposition} \label{pro:WH}
All injective, positive semi-definite self-adjoint bounded linear multiplication operators $M_\lambda$ of type \eqref{eq:Multop} are non-compact.
\end{proposition}


As a consequence 
\begin{equation} A \not \leq_{\B,\B} A^\prime.\end{equation}
It is easy to see that there is a bounded linear operator $S:\ell^2(\NN) \to L^2([0,\infty))$ defined as
\begin{equation} \label{eq:Soperator}
[Sz](\omega):= \sum \limits_{n=1}^\infty z_n \,\chi_{[n-1,n)}(\omega) \quad  (\omega \in [0,\infty), \quad z=(z_1,z_2,\ldots)^T \in \ell^2(\NN)),
\end{equation} 
such that with  $AS: \ell^2(\NN) \to L^2([0,\infty))$ 
\begin{equation} \label{eq:sim2}
 S\,A^\prime \,=\, A \, S \,.
\end{equation}
Again we observe  some kind of similarity between the operators $A^\prime$ and $A$, even if one cannot simply write $A^\prime=S^{-1}AS$. 
But the situation here is different from the situation in Section~\ref{subsec:Hausdorff}, because $\|Sz\|^2_{L^2([0,\infty))}=\|z\|_{\ell^2(\NN)}^2$ for all $z \in \ell^2(\NN)$ and $S$ has a closed range in $L^2([0,\infty))$ such that the Moore-Penrose pseudoinverse operator
$S^\dagger$ 
is bounded and $A^\prime=S^\dagger AS$ holds  true.
Hence, we have as main conclusion 
\[ A^\prime \leq_{\B,\B} A , \quad 
A^\prime \not \sim_{\B,\B} A  
\qquad \text{ but } \sigma(A^\prime) = \sigma(A).
\]
Thus, in the non-compact case a result analogous to Corollary~\ref{cor:iso} cannot hold.



\section{Banach space case} \label{sec:Banach}
In this part, we extend some results to  
the Banach space case.  Since orthogonality does not 
make sense there, we  only use the 
ordering $\leq_{\B,\B}$. 
We denote by a superscript $*$ the topological dual operator.

\begin{theorem}\label{th3}
Let $A:X \to Y$, $A':X'\to Y'$ be bounded operators between Banach spaces. 
Then 
\[  A' \leq_{\B,\B} A  \Longleftrightarrow   {A'}^* \leq_{\B,\B}  {A}^* .\] 
\end{theorem}

For compact operators mapping between Banach spaces, we can 
replace singular values by $s$-numbers:
\begin{definition}
 Let $A,A'$ be compact operators  as above. 
 Let $s$ be an $s$-number \cite{Pietsch}. 
 Define the ordering 
 \[ A' \leq _{s,C} A \Longleftrightarrow  \exists\, C>0 \;\mbox{such that} \,\; \forall n :  \,  s_n(A') \leq C s_n(A). \]
 \end{definition}

Directly from the axioms of $s$-numbers we have the  result 
\begin{lemma}
\begin{equation}\label{simply} 
A' \leq_{\B,\B} A  \Rightarrow A' \leq _{s,C} A.
\end{equation}
\end{lemma} 
It is unknown to the authors if and under what condition
a reverse implication could hold true.

The Douglas range inclusion theorem does not hold in Banach spaces without additional assumptions. 
The following results are due to \cite{Barnes,Embry}: 
\begin{theorem}\label{th_8}
Let $\,A: X \to Y$ and $A':X \to Y'$ be bounded linear operators mapping between the Banach spaces $X$, $Y$, and $Y'$. 
\begin{enumerate}
    \item\label{Baone} The following statements are equivalent:
 \begin{itemize} 
  \item \[ \Ra({A'}^*) \subset \Ra(A^*) . \]
  \item \[ A' = T A , \quad T \in \B(\ol{\Ra(A)},X) . \]
  \item  \[ \exists \, C >0:\quad   \|A'x \| \leq C \| A x\| \quad\forall \,x \in X. \]
 \end{itemize}
\item\label{Batwo} 
Let $Y' = Y$. Moreover, assume that  
\[ \Ra(A') \subset \Ra(A) \]
and  that $\N(A)$ has a closed complemented subspace $X = \N(A) \oplus W$.
Then
\[ A' = A S , \quad S \in \B(X,X) \,. \]
\end{enumerate}
\end{theorem}


As a consequence we have the following implication:
\begin{theorem}
Let $A,A':X\to Y$ and $A$ be injective (or have finite-dimensional 
nullspace). 
Then 
\[  A' \leq_{\B,\B} A   \Longleftrightarrow 
\exists \, T \in \B:   {A'}^* T^* \leq_{norm} A^* .\]
 \end{theorem}
\begin{proof}
The implication ``$\Rightarrow$" follows easily from the 
boundedness of $S^*$. For the opposite direction 
``$\Leftarrow$", use Theorem~\ref{th_8}, item~\ref{Baone}. with 
$A' := A^*T^*$.
\end{proof}

\section{Nonlinear case studies}\label{sec:nonlinear}
\subsection{General assertions} \label{subsec:generalnl}
In this section, we discuss and extend the concept of ordering by ill-posedness for the case of nonlinear 
operator equations
\begin{equation}\label{eq:nlopeq} 
 F(x) = y\,,
\end{equation} 
modeling nonlinear ill-posed problems, where the \emph{nonlinear} operator 
$F$ with\linebreak \mbox{$F:D(F) \subset X \to Y$} is \emph{weakly sequentially closed}
with \emph{non-compact, convex and closed domain} $D(F)$, and $X,Y$ are Hilbert spaces.
Unless stated otherwise, we also assume  that $F$ has a Fr\'echet derivative as well. 

We usually assume that $F$ is {\em locally ill-posed} at a point $\xd \in D(F)$ in the sense that there exist a closed ball
$B_\rho(\xd)$ around $\xd$ with radius $\rho>0$ and a sequence $x_{n} \in D(F) \cap B_\rho(\xd)$ with  
\begin{equation} \label{eq:locill}
\lim_{n \to \infty}\|F(x_n)-F(\xd)\|_Y =0\,,   \quad \mbox{but}  \quad  \|x_n-\xd\|_X \not\to 0 \quad \mbox{as} \quad n \to \infty.
\end{equation}

For characterizing the strength and nature of local ill-posedness at $\xd \in D(F)$, ideas have been outlined in \cite{GorHof94,Hof94,Hof98,HoSch} to use the ill-posedness nature of linearizations to $F$ at $\xd$, and in particular the decay rate of singular values of the compact Fr\'echet derivative $\Fprim{\xd}$. This leads to the following definition:


\begin{definition}[local degree of ill-posedness]\label{def:locdeg}
Let the \emph{local degree of ill-posedness} of the nonlinear operator $F:D(F) \subset X \to Y$ at the point $\xd \in D(F)$ 
be the degree of ill-posedness of the compact Fr\'echet derivative
$\Fprim{\xd}$ at this point. Then moderate and severe ill-posedness of $\Fprim{\xd}$ and the associate interval of ill-posedness 
(see \eqref{intill}) transfer to the local nature of ill-posedness of the nonlinear operator $F$ at $\xd$.
\end{definition}

The generic examples of locally ill-posed operators are the completely continuous (compact) ones: 
\begin{proposition} \label{pro:sufloc}
If $F:D(F) \subset X \to Y$ is completely continuous, then the nonlinear equation \eqref{eq:nlopeq} is locally ill-posed everywhere on the interior ${\rm int}(D(F))$ of the domain $D(F)$.
\end{proposition}
\begin{proof}
Choose $\xd \in {\rm int}(D(F))$ arbitrarily. Then, for a sufficiently small radius $\rho>0$, we have that $B_\rho(\xd) \subset D(F)$, and for any orthonormal system $\{e_n\}_{n=1}^\infty$ in $X$ we have $e_n \rightharpoonup 0$ and $\xd+\rho e_n \rightharpoonup \xd$ as $n \to \infty$ with $x_n:=\xd+\rho e_n \in B_\rho(\xd)$ and $\|x_n-\xd\|_X=\rho>0$.
Then the assume weak continuity yields the assertion of the proposition.
\end{proof}

Moreover, for completely continuous (compact) operators, the Fr\'echet derivative is compact as well, hence the linearization is 
also ill-posed. As the Example~\ref{ex:autocon} below will show, the converse assertion is not true. There exist non-compact nonlinear operators $F$ such that the Fr\'echet derivatives $\Fprim{\cdot}$ are compact linear operators. This illustrates 
the problem that even for differentiable operators and with regard to ill-posedness, 
the nonlinear operator can behave completely differently than its 
 linearization, and in particular Definition~\ref{def:locdeg} might become useless in certain situations.  

Another example illustrating this  appeared in \cite[A.~1]{EKN89}, where a nonlinear operator is everywhere 
ill-posed, but its linearization is well-posed on a dense set! This shows that without additional conditions, 
linear concepts of ill-posedness, as we used in the previous sections, are not appropriate in the nonlinear case. 

The additional conditions that we need here is that all  linearizations locally behave in a similar way: 

\begin{definition}
Let $F$ be a Fr\'echet-differentiable operator that is locally ill-posed at $\xd$. 
We say that it is \emph{stably ill-posed} if for all $x \in \B_\rho(\xd)$ 
\[  F'(x) \sim_{\B,\B} F'(\xd) .\] 
\end{definition}

As a consequence of Theorem~\ref{th0}, a necessary and sufficient condition for stable ill-posedness is that 
there exists constants $\underline c$, $\overline c$ such that 
\begin{equation} \label{eq:asymp}
\underline c \, \sigma_n (\Fprim{x}) \le \sigma_n (\Fprim{\xd}) \le \overline c\, \sigma_n (\Fprim{x}) \qquad 
\forall x \in B_\rho (\xd)\cap D(F)\, .
\end{equation}
Conditions sufficient for \eqref{eq:asymp} are given by certain nonlinearity conditions. Assume that 
for a linear bounded operator $\Rop{x,\tilde x} \in \B(Y,Y)$ depending on $x,\tilde x \in B_\rho(\xd)$ and for some exponent  $0<\kappa \le 1$ with the constant $C_R>0$, it holds that 
\begin{alignat}{2}
&\Fprim{x}=\Rop{x,\tilde x}\,\Fprim{\tilde x}, & \quad & \forall x,\tilde x \in B_\rho(\xd), \label{eq:rotleft} \\ 
 &\|\Rop{x,\tilde x}-I\|_{\B(Y,Y)} \le C_R \,\|x-\tilde x\|_X^\kappa, 
 \label{eq:rotleftnorm} 
 & & 
\end{alignat}
and we refer to \cite{Kaltenbacher08} for consequences and to \cite{Hanke95} for applications.
A similar slightly more general condition
was also used in \cite{Blasch}. 

In the notation of the previous sections, 
condition \eqref{eq:rotleft} means that 
\[ \Fprim{x} \sim_{\B,I} \Fprim{\tilde x}, \qquad\forall x,\tilde x \in B_\rho(\xd)\,,\]
where the operator $\Rop{x,\tilde x}$ constitutes the connecting operator in $\B$. 
  
A direct consequence of previous results (see Table~\ref{tab1}) is the following:   
\begin{proposition} \label{pro:sufasymp}
Condition \eqref{eq:rotleft} implies stable ill-posedness. Moreover, we have 
\[  \eqref{eq:rotleft}  \quad \Longleftrightarrow  \quad \left( \Fprim{x} \sim_{\text{norm},C} \Fprim{\tilde x}, \quad 
\forall x,\tilde x \in B_\rho(\xd) \right) \quad 
\Rightarrow   \quad 
 \eqref{eq:asymp}\, .
\]
\end{proposition}
  
We are now in the position to state different concepts of ordering in the nonlinear case.

 \begin{definition}\label{def8}
 Let Let $F,G$ be nonlinear Fr\'echet-differentiable operators and let $B_\rho(\xd)$ be a ball around $\xd$ with radius $\rho>0$.
Denote by $\NB$ the set of continuous possibly nonlinear operators
defined on a ball $B_\rho$
 (with possibly additional properties stated when required.)
 
Define
 \[ F \leq_{\NB,\NB} G \ \Longleftrightarrow   \
 \begin{array}{l} \Psi \in \NB, \Phi \in \NB:
 \  \text{with } \\[1mm]
 F(x) = \Psi\circ  G \circ \Phi(x) \end{array} \quad \forall x \in B_\rho(x_0).\]

Moreover, for $\xd$ fixed and $F,G$ defined on $\B_\rho(\xd)$,
we define the linearized ordering as 
 \[ F \leq_{\B,\B}^{\Lin} 
 G \Longleftrightarrow  \Fprim{\xd}\leq_{\B,\B} \Gprim{\xd}  \] 
 and the uniform linearized ordering as 
  \[ F \leq _{\B,\B}^{\ULin} G \Longleftrightarrow  \Fprim{x} \leq_{\B,\B}  \Gprim{x} \qquad \forall x \in B_{\rho}(\xd) . \] 
 \end{definition}
Here the nonlinear ordering $\leq_{\NB,\NB}$ reflects the 
idea that the information transfer in $F$  from input to 
output is passed though $G$ and hence any ``information" lost in 
$G$ will also be lost in $F$. Thus, $F$ is considered more ill-posed. 

Clearly, by the chain rule we have 
\[ F \leq _{\NB,\NB} G  
\Rightarrow F \leq _{\B,\B}^{\ULin} G   \Rightarrow  F \leq _{\B,\B}^{\Lin}  G. \]
Moreover if $F$ and $G$ are stably ill-posed, then 
\[ F \leq _{\B,\B}^{\Lin}  \Rightarrow F \leq _{\B,\B}^{\ULin} G.     \]
As the nonlinear ordering is difficult to prove,  the 
linearized orderings serve as practical computable conditions 
to verify it. Thus, it is of interest to consider conditions 
when the linearized orderings imply the nonlinear one. 

Below we further state such conditions. For this we have to take into account the degree of nonlinearity:

\subsection{Degree of nonlinearity}
The nonlinearity conditions \eqref{eq:asymp} 
or \eqref{eq:rotleft}  constitute a class of restrictions on the nonlinearity. 
Yet another (weaker) class of conditions are the tangential cone conditions. In \cite{HoSch} a parametric class of
such conditions were postulated in  \cite[Definition~1]{HoSch} by introducing the concept of a \emph{local degree of nonlinearity} at $\xd$ with the exponent triple  $(\gamma_1,\gamma_2,\gamma_3) \in [0,1]\times [0,1] \times [0,2]$ as follows:
\begin{definition}[local degree of nonlinearity]
We call the operator $F:D(F) \subset X \to Y$ \emph{locally nonlinear at $\xd \in D(F)$ of degree}  $(\gamma_1,\gamma_2,\gamma_3) \in [0,1]\times [0,1] \times [0,2]$ if there is a radius $\rho>0$ and a constant $q>0$ such that the estimate
\begin{equation}\label{eq:degnonli}
\begin{aligned}\|F(x)& - F(\xd) - \Fprim{\xd} (x-\xd) \|_Y  \\ 
&\leq q \; \| \Fprim{\xd} (x-\xd) \|_Y^{\gamma_1} \;\| F(x) - F(\xd) \|^{\gamma_2}_Y \; \|x-\xd\|_X^{\gamma_3} 
 \end{aligned}\qquad \forall x \in B_\rho(\xd)
 \end{equation}
holds true.
\end{definition}

The inequality in \eqref{eq:degnonli} is not scaling invariant in $F$ (i.e., invariant when $F$ is replaced by $\lambda F$ 
for any $\lambda \in \mathbb{R}$). This is, however, achieved when we restrict our considerations to the degree triple
$(1-\sig,\sig,0)$ for $\sig \in [0,1]$ with constants  $q=q_\sig>0$ (subsuming  the $\gamma_3$-part  into the constants)
such that
\begin{equation}\label{tcc3}
\begin{aligned}
\|F(x)& - F(\xd) - \Fprim{\xd} (x-\xd) \|_Y & \\ 
&\leq q_\sig  \;\| \Fprim{\xd} (x-\xd) \|_Y^{1-\sig}\; \| F(x) - F(\xd) \|^{\sig}_Y 
 \end{aligned}\qquad \forall x \in B_\rho(\xd). 
 \end{equation}
Under the exceptional requirements that the constants $q_\sig$ must be less than one for $\sig=0$ and $\sig=1$, the conditions in \eqref{tcc3} are basically equivalent for all  $\sig \in [0,1]$ in the sense that an inequality chain 
$$ \underline K\,\| \Fprim{\xd} (x-\xd) \|_Y \le  \| F(x) - F(\xd) \|_Y  \le \overline K\,\| \Fprim{\xd} (x-\xd) \|_Y\qquad \forall x \in B_\rho(\xd) $$
is valid for all such exponents $\sig$ with existing constants $0<\underline K \le \overline K < \infty$. Vice versa such chain
implies \eqref{tcc3} for all $\sig \in [0,1]$ and appropriate constants $q_\sig>0$ as the following Lemma~\ref{lemxx} shows. We refer also to \cite{HoSch}, where the cases $\sig = 0$ and $\sig=1$ had been discussed in this context.

We mention that the case $\sig = 1$ in \eqref{tcc3} corresponds to the well-known strong tangential cone condition of Hanke, Neubauer, and Scherzer, 
introduced in \cite{Hanke95}. This condition plays a crucial role in the analysis of nonlinear iterative methods. 
Note also that for this purpose the constant $q_1$ has to be assumed small, say $q_1<1$.

\begin{lemma}\label{lemxx}
\ \\[-2ex]
\begin{enumerate}
\item Let  \eqref{tcc3} hold for some $\sig \in [0,1]$,  where in case $\sig = 1$, we have to assume that $q_1 <1$. 
 Then there exists a constant $\overline K>0$ such that 
 \begin{equation}\label{upper} \|F(x) - F(\xd)\|_Y \leq \overline K \,\| \Fprim{\xd} (x-\xd)\|_Y \qquad \forall x \in B_\rho(\xd).  \end{equation}

\item   Let  \eqref{tcc3} hold for some $\sig \in [0,1]$, where in case $\sig = 0$ we have to assume that $q_0 <1$.
Then there also exists a constant $\underline K>0$ such that 
\begin{equation}\label{lower} 
\underline K \,\| \Fprim{\xd} (x-\xd) \|_Y  \leq \|F(x) - F(\xd)\|_Y   \qquad \forall x \in B_\rho(\xd).  \end{equation}
\end{enumerate} 
Conversely, the conditions \eqref{upper} and  \eqref{lower} together imply 
\eqref{tcc3} for any $\sig\in[0,1]$ 
and associated constants $0<q_\sig<\infty$
(although for no $\sig$ necessarily connected with a smallness conditions $q_\sig<1$).  
\end{lemma}
\begin{proof}
Assuming \eqref{tcc3}, the inequality \eqref{upper} is valid by using first
a triangle inequality and then a version of Young's inequality.
The constants $\overline K>0$  can be found 
as the solutions of $z-1 = q_\sig z^\sig$, $z\geq 0$, which exist for $\sig \in [0,1)$ 
and for $\sig=1$ in case that $q_1<1$. 

The inequality  \eqref{lower} follows from  \eqref{tcc3} in an analog manner. 
The constants $\underline K>0$ can be chosen as the inverse of the solution 
to $z-1 = q_\sig z^{1-\sig}$. 

The opposite direction, starting from \eqref{upper} and  \eqref{lower}, needs again the triangle inequality. 
Then, the inequality \eqref{tcc3} is obtained for arbitrary $\sig \in [0,1]$ with appropriate constants $q_\sig>0$  
by either estimating $a \leq {\underline K}^{-1} b$
or $b \leq {\overline K} a$ with $a = \| \Fprim{\xd} (x-\xd) \|$ 
and $b = \| F(x)-F(\xd) \|$ and then using 
that $\min\{a,b\} \leq a^\sig {b}^{1-\sig}$. 
\end{proof}

As stated above the conditions 
\eqref{eq:rotleft}--\eqref{eq:rotleftnorm} imply the tangential cone conditions: 
(see, e.g.,  \cite{Hanke95}).
\begin{proposition}
Let \eqref{eq:rotleft}--\eqref{eq:rotleftnorm} hold. Then \eqref{tcc3} holds in the form  
$$\|F(x)-F(\xd)-\Fprim{\xd}(x-\xd)\|_Y \le \frac{C_R}{1+\kappa}\,\|\Fprim{\xd}(x-\xd)\|_Y\, \|x-\xd\|_X^\kappa.$$
\end{proposition}

\subsection{Autoconvolution as counterexample}
The autoconvolution problem is an important practically relevant example, where many nonlinearity conditions 
fail, and in particular where the definition of the local degree of ill-posedness is hardly meaningful.
A comprehensive analysis of the following autoconvolution operator with respect to ill-posedness and mentioned properties below was started with the seminal paper \cite{GorHof94}. An interesting deficit study concerning usual nonlinearity conditions can be found in \cite{BuerHof15}.

\begin{example} \label{ex:autocon}\rm
Let $X=Y=L^2(0,1)$ and with $D(F)=L^2(0,1)$ the \emph{autoconvolution operator}
\begin{equation} \label{eq:auto}
[F(x)(s):= \int\limits_0^s x(s-t)\,x(t)\,dt \quad (0 \le s \le 1)\,.
\end{equation}
The associated nonlinear operator equation \eqref{eq:nlopeq} with the autoconvolution operator from \eqref{eq:auto} is \emph{locally ill-posed everywhere} on $L^2(0,1)$.
On the whole $L^2(0,1)$ this autoconvolution operator is a \emph{non-compact} nonlinear operator with the Fr\'echet derivative $\Fprim{x} \in \B(L^2(0,1),L^2(0,1))$ as
\begin{equation} \label{eq:Freauto}
[\Fprim{x}\,v](s)= 2 \int \limits_0^s x(s-t)\,v(t)\,dt \quad (0 \le s \le 1, \;v \in L^2(0,1)),
\end{equation}
which is a \emph{compact} linear operator for all $x \in L^2(0,1)$ satisfying a Lipschitz condition
$$\|\Fprim{x}-\Fprim{\xd}\|_{\B(L^2(0,1),L^2(0,1))} \le L\,\|x-\xd\|_{L^2(0,1)} \quad \forall\;x \in L^2(0,1)$$
with global Lipschitz constant $L=2$. This implies in detail
$$ \|F(x)-F(\xd)-\Fprim{\xd}(x-\xd)\|_{L^2(0,1)} = \|F(x-\xd)\|_{L^2(0,1)} \le \|x-\xd\|_{L^2(0,1)}^2$$
for all $x \in L^2(0,1)$. This provides us with a \emph{local degree of nonlinearity} $(0,0,2)$ at $\xd$ everywhere. No tangential cone condition and no degree of nonlinearity $(\gamma_1,\gamma_2,\gamma_3)$ can be shown, where either $\gamma_1$ or $\gamma_2$ are positive for any $\xd \in L^2(0,1)$.
In particular, the nonlinearity condition \eqref{eq:rotleft} fails everywhere and for all $0<\kappa \le 1$, see \cite[Cor.~2.3]{BuerHof15}.
\end{example}

Indeed, the condition \eqref{eq:asymp} for stable ill-posedness 
cannot hold, and much more is true: for a fixed asymptotics  (local degree of ill-posedness) of the singular values 
at $\xd$,  arbitrarily changing degrees may occur in any neighbourhood of $\xd$. That fact can be seen by inspection of the linear convolution operator \eqref{eq:Freauto} and has been outlined in detail in \cite[\S~5]{GorHof94}. This is an indicator for strong instability, and under such circumstances it does not make sense to use the Fr\'echet derivatives as measures for the strength of ill-posedness of $F$ at $\xd$.

\subsection{Nonlinearity conditions and stable ill-posedness}
We now study the consequences of the nonlinearity conditions above, with respect to 
the relation between the nonlinear and the linear ill-posedness. 

The main result in this section is that tangential cone conditions 
\eqref{tcc3} {\em almost} imply a 
factorization into ``nonlinear well-posed'' and ``linear ill-posed'', 
while the stronger condition \eqref{eq:rotleft} removes the ``almost''. 

The next result has already 
been stated in \cite{HoSch} for the  case of injective $\Fprim{\xd}$ and \eqref{tcc3} with $\sig = 0$. 
\begin{theorem}
Let \eqref{upper} and \eqref{lower} hold. 
Then $F$ can be factorized as 
\begin{equation}\label{eq:fac} F(x) - F(\myx) = N\circ \left(\Fprim{\myx} (x-\myx) \right)  \qquad \forall x \in B_\rho(\xd), \end{equation}
where the nonlinear operator 
\[ N: \Ra(\Fprim{\myx}) \cap B_\rho(0)  \to Y,  \qquad \text{ with }\; N(0) = 0, \] 
has the property 
\begin{align}\label{eq:NN} 
 \underline K\,  \|z\|_Y \leq \|N(z)\|_Y \leq \overline K\, \|z\|_Y \qquad \forall z \in 
\Ra(\Fprim{\myx}).
\end{align}

Conversely, if \eqref{eq:fac} holds
with such an $N$ satisfying \eqref{eq:NN},
then \eqref{upper} and \eqref{lower} are satisfied.
\end{theorem}
\begin{proof}
For simplicity we set $A:= \Fprim{\myx}$. 
It follows  from \eqref{upper} that $F(x + \mathbf{n}) = F(x)$, when $\mathbf{n} \in   \N(A)$.  
Let $z \in \Ra(A)$ with $z = A h$ and $h \in \N(A)^\bot$.   
Note that $A^\dagger A = P_{\N(A)^\bot}$, i.e.,
the orthogonal projector onto $\N(A)^\bot$ such that $A^\dagger A h = h$.
Let us define the operator $N$ as 
\[ N(z) := F(\myx + A^\dagger z ) - F(\myx). \]   
Then the properties \eqref{eq:NN} directly follow since 
\[ \|N(z)\|_Y = \|F(\myx + A^\dagger A h) - F(\myx)\|_Y =
 \|F(\myx + h) - F(\myx)\|_Y 
\]
and by taking into account \eqref{upper} and \eqref{lower}. 

Conversely, if \eqref{eq:fac} holds with \eqref{eq:NN}, then 
\[ \|F(\myx + h) - F(\myx)\|_Y = \|N\circ \Fprim{\myx} h\|_Y \sim  \|\Fprim{\myx} h\|_Y, \] 
i.e., \eqref{upper} and \eqref{lower} 
hold. 
\end{proof}

The previous theorem indicates that the tangential cone conditions can 
{\em roughly} be stated as $F \sim_{\NB,I} \Fprim{\myx}$, i.e., the nonlinear operator 
is as ill-posed as its linearization. However, this interpretation is 
not completely correct since 
the mapping $N$ is only defined on the non-closed set $\Ra(\Fprim{\myx})$ and, in 
particular, is {\em not necessarily a continuous} mapping. Thus, $N$ does not 
satisfy the requirements in Definition~\ref{def8}. 

However, by using the stronger version \eqref{eq:rotleft} we can prove that $N$ is continuous
and remove the  ``roughly'' in the previous statement.

\begin{theorem}
 Let $F$ satisfy \eqref{eq:rotleft}. Then for any $\myx$  the factorization  \eqref{eq:fac} holds with 
 $N: \overline{\Ra(\Fprim{\myx}}\cap B_\rho(0)) \to Y $ a Lipschitz continuous operator that locally 
 has a Lipschitz continuous inverse.  
%
\end{theorem}
\begin{proof}
Since  \eqref{eq:rotleft} implies the tangential cone condition, the factorization exists. 
We only have to show continuity. Set again $A:= \Fprim{\myx}$. and let
$z_1,z_2 \in \Ra(A)$ with  $z_1 = A h_1$, $z_2 = A h_2$,  where $h_1,h_2 \in \N(A)^\bot$. 
Then using $\Fprim{x} \sim_{norm} \Fprim{\xd}$, we obtain
\begin{align}\label{eq:Nup} 
\begin{split} \|N(z_1) &- N(z_2)\|_Y = \|F(\myx + A^\dagger A h_1) - F(\myx + (A^\dagger A  h_2)\|_Y  \\
&= 
 \|F(\myx +h_1) - F(\myx + h_2)\|_Y \\
 & \leq \|\Fprim{\myx + h} (h_2-h_1) \|_Y \sim 
 \|\Fprim{\myx} (h_2-h_1) \|_Y  = \|z_1-z_2\|_Y.
 \end{split}
\end{align}
Thus, $N$ is Lipschitz on $\Ra(A)$ and in particular uniformly continuous. 
It follows that $N$ can be extended 
 to a Lipschitz continuous operator on $\overline{\Ra(A)}$ (see, e.g., \cite[p.~190]{Bour}). 
 Since $\|z_1 -z_2\|_Y \leq C \|N(z_1) - N(z_2)\|_Y$, $N$ is injective, and the local inverse $N^{-1}$
 exists and is Lipschitz continuous. 
%
\end{proof}

This means that \eqref{eq:rotleft} implies that 
\[ F \sim_{\NB,I} \Fprim{\myx}. \]

Finally, under the conditions  \eqref{eq:rotleft} for $F$ and $G$ we obtain 
the result that the nonlinear ordering is equivalent to the linearized ordering: 
\begin{theorem}
Assume that   $F$ and $G$ are Fr\'echet-differentiable 
and defined  on $\B_\rho(\xd)$
and both 
satisfy \eqref{eq:rotleft} with certain operators $\Rop{1}$, $\Rop{2}$. 
Then 
 \[ F \leq _{\B,\I}^{\Lin} G \Rightarrow  F \leq_{\NB,I} G .\]
\end{theorem}
\begin{proof}
We have that $F(x) = N_1 \circ \Fprim{\xd}(x-\xd)$  
$G(x) = N_2 \circ \Gprim{\xd}(x-\xd)$. 
Hence, with $T \in \B$,
\[  N_1^{-1}(F(\xd+h) -F(\xd)  = \Fprim{\xd}h  = T \Gprim{\xd}h = T N_2^{-1}(G(\xd+h)-G(\xd)). \] 
Thus, 
\[ F(\xd + h) - F(\xd) = N_1\left(T N_2^{-1}(G(\xd+h)-G(\xd)) \right), \] 
which gives  $F \leq_{\NB,I} G$.
\end{proof}

\end{document}